\documentclass{article}
\usepackage[utf8]{inputenc}

\usepackage[english]{babel}

\usepackage[utf8]{inputenc}
\usepackage[T1]{fontenc}

\usepackage{amssymb}
\usepackage{stmaryrd}
\usepackage{amsmath,amsfonts}
\usepackage[tikz]{bclogo}
\usepackage[top=4.34cm]{geometry}

\usepackage{lmodern}
\usepackage{textcomp}
\usepackage{mathtools, bm}
\usepackage{amssymb, bm}
\usepackage[unq]{unique}
\usepackage{enumerate}
\usepackage{comment}
\usepackage[breaklinks]{hyperref}

\usepackage[numbers]{natbib}  

\usepackage[breaklinks]{hyperref}
\usepackage[numbers]{natbib}
\usepackage{graphicx} 

\usepackage{amsthm}

\usepackage{tikz}
\usepackage{caption}
\usepackage{subcaption}

\usepackage{cleveref}

\newtheorem{theorem}{Theorem}[section]
\newtheorem{lemma}[theorem]{Lemma}

\newtheorem{claim}[theorem]{Claim}

\let\SS\relax
\newcommand{\SS}{\mathbb{S}}
\newcommand{\cD}{\mathcal{D}}

\newcommand{\cO}{\mathcal{O}}

\newcommand{\defined}{\mathrel{\coloneqq}}
\undef{\set}
\DeclarePairedDelimiter{\set}{\lbrace}{\rbrace}
\newcommand{\eps}{\varepsilon}

\newcommand{\inter}{\mathbin{\cap}}

\newcommand{\st}{\mathbin{\colon}}
\newcommand{\from}{\colon}

\newcommand{\<}{\langle}
\renewcommand{\>}{\rangle}

\undef{\abs}
\DeclarePairedDelimiterX{\abs}[1]
  {\lvert}{\rvert}{\ifblank{#1}{\,\cdot\,}{#1}}
\undef{\norm}
\DeclarePairedDelimiterX{\norm}[1]
  {\lVert}{\rVert}{\ifblank{#1}{\,\cdot\,}{#1}}

\newcommand{\disc}{\cD^d}

\title{A note on high-dimensional discrepancy of subtrees}

\author{Lawrence Hollom \footnote{Department of Pure Mathematics and Mathematical Statistics (DPMMS), University of Cambridge, Wilberforce Road, Cambridge, CB3 0WA, United Kingdom, \href{mailto:lh569@cam.ac.uk}{lh569@cam.ac.uk}} 
\and Lyuben Lichev \footnote{Institute of Science and Technology Austria (ISTA), 3400 Klosterneuburg, Austria, \href{mailto:lyuben.lichev@ist.ac.at}{lyuben.lichev@ist.ac.at}} 
\and Adva Mond \footnote{Department of Mathematics, King’s College London, Strand, London, WC2R 2LS, United Kingdom, \href{mailto:adva.mond@kcl.ac.uk}{adva.mond@kcl.ac.uk}} 
\and Julien Portier \footnote{Department of Pure Mathematics and Mathematical Statistics (DPMMS), University of Cambridge, Wilberforce Road, Cambridge, CB3 0WA, United Kingdom, \href{mailto:jp899@cam.ac.uk}{jp899@cam.ac.uk}}}

\date{\today}

\begin{document}

\maketitle

\begin{abstract}
For a tree $T$ and a function $f\from E(T)\to \SS^d$, the imbalance of a subtree $T'\subseteq T$ is given by $|\sum_{e \in E(T')} f(e)|$.
The $d$-dimensional discrepancy of the tree $T$ is the minimum, over all functions $f$ as above, of the maximum imbalance of a subtree of $T$.
We prove tight asymptotic bounds for the discrepancy of a tree $T$, confirming a conjecture of Krishna, Michaeli, Sarantis, Wang and Wang.
We also settle a related conjecture on oriented discrepancy of subtrees by the same authors.
\end{abstract}

\section{Introduction}

Discrepancy theory is a modern and dynamic area of research within mathematical optimisation and algorithm design, focused on the study of imbalanced structures in weighted or coloured mathematical objects. 
Originating with Weyl’s foundational work~\cite{Wey16}, the field quickly developed rich connections with number theory, combinatorics, ergodic theory, and discrete geometry. 
We recommend the excellent books by Beck and Chen~\cite{BC87}, Matoušek~\cite{Mat99}, and Chazelle~\cite{Cha00}, which provide a comprehensive introduction to discrepancy theory.

The study of discrepancy in graphs was pioneered by two foundational works: Erd\H{o}s and Spencer~\cite{Erdos1971} studied the discrepancy of $2$-colourings of the complete $k$-uniform hypergraph $K_n^k$ with respect to its subcliques, and a subsequent paper by Erd\H{o}s, F\"{u}redi, Loebl and S\'{o}s \cite{EFLS95} studied the discrepancy of $2$-colourings of the complete graph $K_n$ with respect to the copies of a particular $n$-vertex tree.
Formally, given a host graph $G$, a subgraph $H \subseteq G$ with $e$ edges, and an $r$-colouring $f\from E(G)\to [r]$, the (combinatorial) discrepancy of $H$ with respect to the colouring $f$ is defined as 
\begin{align}
\label{eq:DefDiscrepancy}
    \mathcal{D}_f(H)=\max_{H' \subseteq G:H' \cong H} \; \max_{i \in [r]} \; \left||f^{-1}(i) \cap E(H')|-\frac{e}{r} \right|,
\end{align}
where the first maximum ranges over all isomorphic copies $H'$ of $H$ in $G$. 
In turn, the (combinatorial) discrepancy of $G$ with respect to the copies of $H$ in $G$ is defined as the minimum, over all $r$-colourings $f$, of $\cD_f(H)$.

Recent years have seen a surge of interest in discrepancy within edge-coloured graphs.
Notably, Balogh, Csaba, Jing, and Pluhár~\cite{BCJP20} explored the discrepancy with respect to spanning trees and Hamilton cycles in 2-coloured graphs with high minimum degree. 
Their work was later extended to $r$ colours independently by Freschi, Hyde, Lada, and Treglown~\cite{FHLT21} and by Gishboliner, Krivelevich, and Michaeli~\cite{GKM22b}. 
Furthermore,~\cite{GKM22b} generalised the results on spanning trees from~\cite{BCJP20} by establishing weak connectivity conditions in multi-coloured graphs sufficient to guarantee a spanning tree with high discrepancy.
Very recently, the discrepancy of copies of a particular spanning tree in $r$-coloured dense graphs $G$ and related optimisation problems were studied by the authors of this paper~\cite{hollom2024discrepancies}.

Gishboliner, Krivelevich, and Michaeli~\cite{GKM22a} also examined $r$-colourings in the Erd\H{o}s-Rényi random graph $G(n,p)$ for values of $p$ just above the Hamiltonicity threshold $p_{\mathrm{Ham}} = (\log n + \log \log n)/n$ showing that, with high probability, for any $r$-colouring of $G(n,p)$ where $n(p - p_{\mathrm{Ham}}) \to \infty$, there exists a Hamilton cycle with at least $(2/(r+1) - o(1))n$ edges in a single colour. 
Moreover, they established a similar result for perfect matchings as well. 
In the 2-colour case, Bradač~\cite{Bra22} established a minimum degree threshold for containment of an $r$-th power of a Hamilton cycle with linear discrepancy.
In a similar setting, Balogh, Csaba, Pluhár, and Treglown~\cite{BCPT21} exhibited a minimum degree threshold for the existence of a $K_r$-factor with linear discrepancy, thus obtaining a discrepancy version of a theorem by Hajnal and Szemerédi \cite{hajnalSzemeredi}.
Bradač, Christoph, and Gishboliner~\cite{BCG2023} further generalised these results to general $H$-factors.
The discrepancy of tight Hamilton cycles in $k$-uniform hypergraphs was recently studied by Gishboliner, Glock, and Sgueglia~\cite{GGS23}.

\vspace{1em}

In this note, we are interested in a setting introduced by Krishna, Michaeli, Sarantis, Wang and Wang~\cite{KMSWW23}.
Given a tree $T$, the \emph{$d$-dimensional discrepancy} of $T$ relative to its subtrees is defined as
\begin{align}
\label{eq:d-dimensional-discrepancy-definition}
    \disc(T)\defined\min_{f\from E(T)\to \SS^d} \;\;\; \max_{T'\text{ subtree of }T} \; \Bigg\lVert \sum_{e\in E(T')} f(e) \Bigg\rVert,
\end{align}
where $\norm{}$ is the Euclidean norm, as it will be for the rest of this paper.\footnote{Note that the maximum in~\eqref{eq:d-dimensional-discrepancy-definition} is attained by continuity of the function 
\[(x_e)_{e\in T}\mapsto \max_{T'\text{ subtree of }T} \; \Bigg\lVert \sum_{e\in E(T')} x_e \; \Bigg\rVert\]
and compactness of the space $(\mathbb S^d)^{|E(T)|}$.
}
For convenience, we will slightly abuse notation by often identifying trees with their edge sets. 
Moreover, we write $T'\subseteq_c T$ to denote that $T'$ is a connected subset of edges of $T$, that is, a subtree of $T$.

We remark that the $d$-dimensional discrepancy of trees with respect to their subtrees has tight connections with the combinatorial discrepancy of $(d+1)$-colourings in the same setting.
More precisely, consider $d+1$ points $v_1,\ldots,v_{d+1}\in \mathbb S^d$ forming a regular simplex. Then, for any two distinct vertices $v_i$ and $v_j$, we have $\<v_i,v_j\>=-1/d$. 
Therefore, for all integers $a_1, \dots, a_{d+1}$ (corresponding to the number of edges in each colour), and for every $i\in [d+1]$, we have 
$$\<a_1v_1+\dots+a_{d+1}v_{d+1},v_i\>=\frac{d+1}{d} a_i-\frac{S}{d},$$ 
where $S=a_1+\dots+a_{d+1}$. 
As a result, 
$$\norm{a_1v_1+\dots+a_{d+1}v_{d+1}}^2= \sum_{i=1}^{d+1} a_i\bigg(\frac{d+1}{d} a_i-\frac{S}{d}\bigg)= \frac{1}{d}\sum_{j<i} (a_i-a_j)^2,$$ 
and therefore large $d$-dimensional discrepancy implies large combinatorial discrepancy.

Let $B$ denote the Beta function defined by $B(z_1,z_2)=\Gamma(z_1)\Gamma(z_2)/\Gamma(z_1+z_2)$, where $\Gamma$ is the Gamma function.
One of the main results from~\cite{KMSWW23} provides the following lower bound on the $d$-dimensional discrepancy of trees.

\begin{theorem}[Theorem 7 in \cite{KMSWW23}]
\label{thm:LBKrishna}
For every $d\in \mathbb N$ and every tree $T$ with $\ell$ leaves,
$$\disc(T) \geq \dfrac{\ell}{d\cdot B(\frac{d}{2},\frac{1}{2})}.$$
\end{theorem}

Krishna, Michaeli, Sarantis, Wang and Wang conjectured that the bound in \Cref{thm:LBKrishna} is essentially tight for $d=1$.
We prove that this bound is actually asymptotically tight for every $d \in \mathbb{N}$ and large number of leaves, thereby confirming their conjecture.
In the following result, and throughout this paper, all asymptotic $o(\cdot)$ notation is with respect to $\ell \rightarrow \infty$.

\begin{theorem}
\label{thm:multidimensional-subtree-discrepancy}
For every $d\in \mathbb N$ and every tree $T$ with $\ell$ leaves,
$$\disc(T) = (1+o(1))\dfrac{\ell}{d\cdot B(\frac{d}{2},\frac{1}{2})}.$$
\end{theorem}

Krishna, Michaeli, Sarantis, Wang and Wang also introduced the concept of oriented discrepancy.
For two trees $T'\subseteq_c T$ and orientations $\sigma,\sigma'$ on the edges of $T,T'$, respectively, define the function $d_{T,T'}^{\sigma, \sigma'}\from T'\to \{-1,1\}$ which outputs 1 for edges $e\in T'$ where $\sigma$ and $\sigma'$ agree, and $-1$ otherwise.
For a tree $T$, let $\cO(T)$ denote the set of all pairs $(T',\sigma')$ of a tree $T'\subseteq_c T$ rooted at a vertex $r$ and an orientation $\sigma'$ away from $r$.

The \emph{oriented discrepancy} of the tree $T$ is defined as
\begin{align*}
    \vec{\cD}(T) \defined \min_{\sigma\text{ orientation of }T}\;\;\;\max_{(T',\sigma') \in \cO(T)} \Bigg|\sum_{e \in T'} d_{T,T'}^{\sigma, \sigma'}(e)\Bigg|.
\end{align*}
Intuitively, oriented discrepancy measures to what extent all rooted oriented subtrees of $T$ can be balanced simultaneously when the edges of $T$ are oriented optimally.
Krishna, Michaeli, Sarantis, Wang and Wang proved the following result.

\begin{theorem}[Theorem 4 in \cite{KMSWW23}]\label{thm:orient}
For every tree $T$ with at least $3$ vertices and $\ell$ leaves,
$$\left \lceil\frac{\ell}{2} \right \rceil +1 \leq \vec{\cD}(T) \leq \ell.$$
\end{theorem}

Our second main result confirms their conjecture that the lower bound in Theorem~\ref{thm:orient} is asymptotically tight.

\begin{theorem}
\label{thm:SignedDiscrepancy}
For every tree $T$ with $\ell$ leaves,
$$\vec{\cD}(T) = \left(\frac{1}{2}+o(1)\right)\ell.$$
\end{theorem}


\section{Proofs}

We start with a lemma proving \Cref{thm:multidimensional-subtree-discrepancy} in the case of stars.
For $\eps>0$, an $\eps$-net on $\mathbb S^d$ is a set of points $E_{\eps}$ such that every point in $\mathbb S^d$ is at Euclidean distance at most $\eps$ from a point in $E_{\eps}$. Also, define
\begin{align}\label{def:phi}
\phi(\ell) \defined \dfrac{\ell}{dB(\frac{d}{2},\frac{1}{2})} + 2\ell^{3/4}.
\end{align}

For every $\ell\in \mathbb N$, denote by $S_{\ell}$ the star with $\ell$ leaves.

\begin{lemma}
\label{lem:MultidimensionalDiscrepancyStar}
For every $d\in \mathbb N$, there exists $\ell_0 = \ell_0(d)\in \mathbb N$ such that, for every $\ell\ge \ell_0$, we have 
$$\disc{(S_{\ell})} \leq \phi(\ell).$$
In particular, we have
$$\disc(S_\ell) = (1+o(1))\dfrac{\ell}{d B(\frac{d}{2},\frac{1}{2})}.$$
\end{lemma}
\begin{proof}
For convenience, we assume that $\ell$ is even; the case of odd $\ell$ follows similarly.
Select a set of $\ell/2$ points $x_1, \dots, x_{\ell/2}$ uniformly at random on the sphere $\mathbb{S}^d$, and define the set 
\begin{align*}
K\defined \set{x_1, -x_1, x_2, -x_2, \dots, x_{\ell/2}, -x_{\ell/2}}.
\end{align*}
We will show that, with high probability, for every subset $K' \subseteq K$, the sum of the elements of $K'$ is at most $\phi(\ell)$.
To this end, for every $a\in \mathbb{S}^d$, define $R_a\defined \sum_{i=1}^{\ell/2} \abs{\<x_i\cdot a\>}$, where $\<\cdot,\cdot\>$ stands for the $(d+1)$-dimensional scalar product.
For $\eps > 0$, we fix an $\eps$-net $E_{\eps}$ on $\mathbb S^d$
of size at most $(1+2/\eps)^{d+1}$ whose existence is guaranteed, for instance, by Lemma 5.2 in \cite{vershynin2010introduction}.

\begin{claim}\label{cl:IneqEpsNet}
Fix $\eps = 1/\ell^2$. Then, there exists $\ell_0 = \ell_0(d)\in \mathbb N$ such that, for every $\ell\ge \ell_0$, with probability at least $1/2$,
for every $a \in E_{\eps}$, we have
\begin{align*}
R_a \leq \dfrac{\ell}{dB(\frac{d}{2},\frac{1}{2})}+{\ell}^{3/4}.
\end{align*}
\end{claim}
\begin{proof}
Fix $a \in E_{\eps}$.
As the $d$-dimensional area of the unit sphere $\mathbb S^{d-1}$ is given by $2\pi^{d/2}/\Gamma(d/2)$, one can compute (see also the proof of Theorem~7 in \cite{KMSWW23}) that, for $x$ sampled uniformly at random in $\mathbb S^{d-1}$, we have
\begin{align*}
\mathbb{E}\left[\left|\<x \cdot a\>\right|\right] = \frac{2\Gamma(\tfrac{d+1}{2})}{\sqrt{\pi}\Gamma(\tfrac{d}{2})}\int_{0}^{\pi/2} (\sin\theta)^{d-1}\cos\theta\; d\theta = \dfrac{2}{d B(\frac{d}{2},\frac{1}{2})}.
\end{align*} 
Hence, by a standard Chernoff bound,
\begin{align*}
\mathbb{P} \left( \biggl|R_a - \dfrac{\ell}{dB(\frac{d}{2},\frac{1}{2})} \biggr| \geq \ell^{3/4} \right) \leq e^{-\ell^{1/2}/4}.
\end{align*}
As $|E_{\eps}| = O(\ell^{2d+2})$, by choosing $\ell_0$ sufficiently large, taking a union bound over all $a \in E_{\eps}$ concludes the proof.
\end{proof}

Finally, conditionally on the event from Claim~\ref{cl:IneqEpsNet} with $\eps = 1/\ell^2$, we have
\begin{equation}\label{eq:split}
\disc(S_{\ell})\le \max_{\eps_1,\ldots,\eps_{\ell/2}\in \{-1,1\}} \bigg\lVert\sum_{i=1}^{\ell/2} \eps_i x_i\bigg\rVert\le \max_{a\in \mathbb S^d} R_a\le \max_{a'\in E_{\eps}} R_{a'} + \max_{a\in \mathbb S^d}\min_{a'\in E_{\eps}} |R_a - R_{a'}|.
\end{equation}
Note that the first term is at most $\frac{\ell}{dB(d/2,1/2)}+\ell^{3/4}$, and the second term is at most $\eps \ell/2$.
Therefore, we have 
$$\disc(S_{\ell}) \leq \frac{\ell}{dB(\frac{d}{2},\frac{1}{2})}+\ell^{3/4} + \frac{\eps \ell}{2} \leq \phi(\ell),$$
which finishes the proof of the first part of statement.

Together with the lower bound from \Cref{thm:LBKrishna}, the second part of the statement is implied.
\end{proof}

We are now ready to prove \Cref{thm:multidimensional-subtree-discrepancy}.

\begin{proof}[Proof of \Cref{thm:multidimensional-subtree-discrepancy}]
Recall that our asymptotic notation is always with respect to $\ell\to \infty$.
Fix $c_d \defined (d B(d/2,1/2))^{-1}$ and a tree $T$ with a set of leaves $L$, and let $P\defined\set{v\in T\st \abs{N(v)\inter L}\geq 1}$ be the set of vertices adjacent to at least one leaf.
The next claim implies that it is sufficient to show Theorem~\ref{thm:multidimensional-subtree-discrepancy} when every vertex in $P$ has degree at least 3.

\begin{claim}
\label{cl:NoParentDeg2}
If there is a leaf $x$ with neighbour $y$ of degree $2$ in $T$, the subtree $T' = T\setminus \{xy\}$ satisfies $\disc(T') = \disc(T)$.
\end{claim}
\begin{proof}
First, we remark that the inequality $\disc(T')\le \disc(T)$ always holds for trees $T'\subseteq_c T$, so it suffices to prove the inequality $\disc(T') \geq \disc(T)$.
Fix $x,y$ as in the statement of the claim and let $z$ be the other neighbour of $y$ in $T$.
Fix a function $f'\from T'\to\SS^d$ realising the minimum in the definition of $\disc(T')$, and define the function $f\from T\to \SS^d$ which coincides with $f'$ on $T'$ and $f(xy)=-f(yz)$. We will show that
\begin{align}
\label{eq:subtree-discrepancy-target}
\max_{\hat T\subseteq_c T'} \bigg\lVert\sum_{e\in \hat{T}} f'(e) \bigg\rVert = \max_{\hat T\subseteq_c T} \bigg\lVert\sum_{e\in \hat T} f(e) \bigg\rVert.
\end{align}
Indeed, if $\hat T\subseteq_c T'$, then $\sum_{e\in \hat{T}} f'(e) = \sum_{e\in \hat T} f(e)$. 
Moreover, if $\hat T\subseteq_c T$ contains the edge $xy$, then
\[\sum_{e\in \hat T} f(e) = \sum_{e\in \hat T\setminus \{xy,yz\}} f(e) = \sum_{e\in \hat T\setminus \{xy,yz\}} f'(e).\]
This confirms~\eqref{eq:subtree-discrepancy-target} and finishes the proof.
\end{proof}
Moving on with our proof of \Cref{thm:multidimensional-subtree-discrepancy}, we partition $T$ into a star forest $F$ and a tree $R$ as follows. 
Initially, we assign to $R$ every edge of $T$ that is not incident to a leaf.
For edges incident to leaves, we decide if they belong to $F$ or $R$ as follows.
For a vertex $p\in P$, let $I_p=\set{px\in T\st x\in L}$ be the set of edges connecting $p$ to $L$.    
\begin{itemize}
        \item If $p$ has at least two neighbours in $V(T)\setminus L$, we add all edges in $I_p$ to $F$.
        \item If $p$ has at most one neighbour in $V(T)\setminus L$, we select $e\in I_p$ arbitrarily (from $\abs{I_p}\geq 2$ possibilities, as $p$ has degree at least 3) and add $I_p\setminus{\{e\}}$ to $F$ and $e$ to $R$.
\end{itemize}
Let $\ell_F = |F|$ and $\ell_R$ be the number of leaves in $R$.
Note that every edge in $F$ is incident to a leaf in $T$, and the only edges in $R$ incident to leaves of $R$ are those which were incident to leaves in $T$.
In particular, $\ell_F+\ell_R=\ell$.
Additionally, by our choice of $R$, it is clear that $\ell_R\leq\ell/2$.

Let $F_1,\dotsc,F_r$ be the connected components of $F$. By definition of discrepancy and the triangle inequality, we have that
\begin{align}
\label{eq:discrepancy-partition}
\disc(T)\leq \disc(R)+\sum_{i=1}^r\disc(F_i).
\end{align}
Recalling~\eqref{eq:split} applied for $F_1\cup \ldots\cup F_r$ instead of $S_{\ell}$ and $\ell_F\ge \ell/2$ instead of $\ell$, as long as $\ell_F\ge \ell_0$ (with $\ell_0$ defined in Lemma~\ref{lem:MultidimensionalDiscrepancyStar}), we have that 
\begin{align}\label{eq:bound-by-star}
\sum_{i=1}^r\disc(F_i)\leq \max_{a\in \mathbb S^d} R_a \leq \phi(\ell_F).
\end{align}
Note that, if $\ell_F\le \ell_0$, the middle term in~\eqref{eq:bound-by-star} is still at most $\ell_F$.

Now, by Claim~\ref{cl:NoParentDeg2}, there is a subtree $R'\subseteq_c R$ of $R$ without parents of leaves of degree 2 and $\disc(R') = \disc(R)$. 
Combining this observation with~\eqref{eq:discrepancy-partition} and \eqref{eq:bound-by-star}, we deduce that, if $\ell_F\ge \ell_0$, $\disc(T)\le \disc(R')+\phi(\ell_F)$.
As a result, an immediate induction shows that, for some sequence $\ell_1,\ldots,\ell_k$ where $\ell_1+\ldots+\ell_k=\ell$, $\ell_i\ge 2\ell_{i+1}$ for all $i\in [k-1]$ and $\ell_{k-1} > \ell_0 \ge \ell_k$, we have
\begin{align}
\label{eq:LastSumInduction}
    \disc(T)\le \bigg(\sum_{i=1}^{k-1} \phi(\ell_i)\bigg) + \ell_k.
\end{align}
Moreover, since $\phi(\ell_i) = c_d \ell_i + 2\ell_i^{3/4}$ by~\eqref{def:phi}, 
and the sum of $\phi(\ell_i)$ over all $\ell_i\le \sqrt{\ell}$ is of order $O(\sqrt{\ell})$,
the sum on the right-hand side in \eqref{eq:LastSumInduction} is equal to $(1+o(1)) c_d \ell$, as desired.
\end{proof}

We now proceed to prove \Cref{thm:SignedDiscrepancy}. Since the proof closely mirrors that of \Cref{thm:multidimensional-subtree-discrepancy}, the justification of some analogous statements is omitted.

\begin{proof}[Proof of \Cref{thm:SignedDiscrepancy}]
First of all, by orienting $\lceil \ell/2\rceil$ of the edges of $S_{\ell}$ towards the central vertex and the remainder away from it, it follows that $\vec{\cD}(S_d)\le \lceil \ell/2\rceil + 1$ (that is, the lower bound from Theorem~\ref{thm:orient} is tight for stars).
Next, we show an analogue of \Cref{cl:NoParentDeg2} in the setting of oriented discrepancy.

\begin{claim}\label{cl:thm1.4}
For every tree $T$ with $\ell$ leaves, there is a subtree $T'\subseteq_c T$ with at most $\ell$ leaves whose parents all have degree at least $3$ and $\vec{D}(T')+1\ge \vec{D}(T)$.
\end{claim}
\begin{proof}
Define $T'$ by iteratively deleting the leaves of $T$ whose unique neighbour has degree 2, and let $\sigma'$ be the orientation of $T'$ minimising 
\[\max_{(\hat T',\hat \sigma') \in \cO(T)} \Bigg|\sum_{e \in \hat T'} d_{T', \hat T'}^{\sigma', \hat\sigma'}(e)\Bigg|.\]
We now construct an orientation $\sigma$ of $T$. 
To begin with, we let $\sigma$ coincide with $\sigma'$ on $T'$. 
We run the iterative construction of $T'$ from $T$ backwards.
Let $xy$ be an edge added at a given step where $x$ is a leaf and $y$ is a vertex of degree 2 with neighbours $x,z$ in the new tree.
In this case, we orient the edge $xy$ towards $y$ if and only if $zy$ is oriented towards $y$.
Finally, we define $\sigma$ as the orientation of $T$ obtained at the end of the process.
 
We say that a path in $T$ is \emph{terminal} if it contains only vertices of degree 2 and ends with a leaf.
Fix $\hat T\subseteq_c T$ with orientation $\hat\sigma$ away from a vertex $r\in V(\hat T)$.
We will construct a tree $\hat T'\subseteq_c T'$ with orientation $\hat\sigma'$ away from some root vertex satisfying 
\begin{align}\label{eq:sum1.4}
\bigg(\sum_{e \in \hat T'} d_{T',\hat T'}^{\sigma', \hat\sigma'}(e)\bigg) + 1\ge \sum_{e \in \hat T} d_{T,\hat T}^{\sigma, \hat\sigma}(e).
\end{align}
The construction goes as follows.
If $r$ is a non-leaf vertex in $T'$, by cutting terminal paths of even length from $T$, one can ensure that there is a tree $\hat T'\subseteq_c T'\cap \hat T$ with orientation $\hat\sigma'$ that agrees with $\hat\sigma$ on $\hat T'$ and such that
\begin{align*}
\sum_{e \in \hat T'} d_{T',\hat T'}^{\sigma', \hat\sigma'}(e) = \sum_{e \in \hat T} d_{T,\hat T}^{\sigma, \hat\sigma}(e).
\end{align*}
If $r$ is a leaf in $T'$, then, similarly to the previous case, one may cut terminal paths from $T$ which all have even length, with one exception: we cut the terminal path starting at $r$ irrespective of its length.
As a result, we obtain a new tree $\hat T'$ with orientation $\hat\sigma'$ satisfying
\begin{align*}
\sum_{e \in \hat T'} d_{T',\hat T'}^{\sigma', \hat\sigma'}(e)-\sum_{e \in T''} d_{T,\hat T}^{\sigma, \hat\sigma}(e)\in \{-1,0,1\}.
\end{align*}
This shows~\eqref{eq:sum1.4} and finishes the proof in this case.

Finally, suppose that $r$ is not in $T'$. 
Then, $r$ must be the endpoint of some subpath $zwr$ of a terminal path in $T$ starting from $z$.
Then, modifying the orientation $\hat\sigma$ by re-rooting $\hat T$ at $z$ does not change the right hand side of~\eqref{eq:sum1.4}, since in $\sigma$, both edges incident to $w$ are either simultaneously oriented towards $w$ or simultaneously oriented away from $w$. 
Iterating this argument takes us to some of the previous two cases and finishes the proof.
\end{proof}

Finally, by combining the analogue of Lemma~\ref{lem:MultidimensionalDiscrepancyStar} for oriented discrepancy, Claim~\ref{cl:thm1.4} and the decomposition $T=F\cup R$ constructed in the proof of Theorem~\ref{thm:multidimensional-subtree-discrepancy},
one can derive the inequality
\[\vec{\cD}(T)\le (\vec{\cD}(R')+1)+\bigg(\left\lceil \frac{\ell_F}{2}\right\rceil+1\bigg),\]
where $R'$ is the tree obtained from $R$ in Claim~\ref{cl:thm1.4}.
As a result, immediate induction shows that, for some sequence $\ell_1,\ldots,\ell_k$ where $\ell_1+\ldots+\ell_k=\ell$ and $\ell_i\ge 2\ell_{i+1}$ for all $i\in [k-1]$, we have
\[\vec{\cD}(T)\le \sum_{i=1}^k \left(\left\lceil \frac{\ell_i}{2}\right\rceil+2\right).\]
Since $k\le \log_2 \ell$, the sum on the right-hand side in the last display is equal to $(1/2+o(1)) \ell$, finishing the proof.
\end{proof}

\bibliographystyle{abbrvnat}  
\bibliography{bib}

\end{document}